\documentclass[12pt]{article}%
\usepackage{graphicx}
\usepackage[intlimits]{amsmath}
\usepackage{latexsym}
\usepackage{amsfonts}
\usepackage{amssymb}%
\makeatletter
\newfont{\footsc}{cmcsc10 at 8truept}
\newfont{\footbf}{cmbx10 at 8truept}
\newfont{\footrm}{cmr10 at 10truept}
\newtheorem{theorem}{Theorem}

\newtheorem{corollary}[theorem]{Corollary}

\newtheorem{problem}[theorem]{Problem}
\newtheorem{proposition}[theorem]{Proposition}

\newtheorem{remark}[theorem]{Remark}

\newenvironment{proof}[1][Proof]{\noindent{\textbf {#1}  }}  {\hfill$\Box$\bigskip}

\begin{document}

\title{\textbf{The trace norm of }$r$\textbf{-partite graphs and matrices}}
\author{V. Nikiforov\thanks{Department of Mathematical Sciences, University of
Memphis, Memphis TN 38152, USA; \textit{email: vnikifrv@memphis.edu}}}
\maketitle

\begin{abstract}
The trace norm $\left\Vert G\right\Vert _{\ast}$ of a graph $G$ is the sum of
its singular values, i.e., the absolute values of its eigenvalues. The norm
$\left\Vert G\right\Vert _{\ast}$ has been intensively studied under the name
of \textit{graph energy}, a concept introduced by Gutman in 1978.

This note studies the maximum trace norm of $r$-partite graphs, which raises
some unusual problems for $r>2$. It is shown that, if $G$ is an $r$-partite
graph of order $n,$ then
\[
\left\Vert G\right\Vert _{\ast}<\frac{n^{3/2}}{2}\sqrt{1-1/r}+\left(
1-1/r\right)  n.
\]
For some special $r$ this bound is asymptotically tight: e.g., if $r$ is the
order of a real symmetric conference matrix, then, for infinitely many $n,$
there is a graph $G\ $of order $n$ with%
\[
\left\Vert G\right\Vert _{\ast}>\frac{n^{3/2}}{2}\sqrt{1-1/r}-\left(
1-1/r\right)  n.
\]
\bigskip



\end{abstract}

\section{Introduction}

The \emph{trace norm} $\left\Vert A\right\Vert _{\ast}$ of a matrix $A$ is the
sum of the singular values of $A;$ it is also known as the \emph{nuclear norm}
or the \emph{Schatten }$1$\emph{-norm of }$A$\emph{.} The trace norm of the
adjacency matrix of graphs has been much studied under the name \emph{graph
energy,} a concept introduced by Gutman in \cite{Gut78}; for an overview of
this vast research, see \cite{GLS12}. Thus, let us write $\left\Vert
G\right\Vert _{\ast}$ for the trace norm of the adjacency matrix of a graph
$G,$ and note that $\left\Vert G\right\Vert _{\ast}$ is just the sum of the
absolute values of the $G$ eigenvalues.

Koolen and Moulton \cite{KoMo01} studied the maximum trace norm of graphs of
order $n$; in particular, they proved that if $G$ is a graph of order $n,$
then
\begin{equation}
\left\Vert G\right\Vert _{\ast}\leq n^{3/2}/2+n/2, \label{KM}%
\end{equation}
with equality if and only if $G$ belongs to a certain family of strongly
regular graphs; in \cite{Hae08} Haemers showed that these graphs arise from a
class of Hadamard matrices. Furthermore, Koolen an Moulton \cite{KoMo03}
proved that if $G$ is a bipartite graph of order $n,$ then
\begin{equation}
\left\Vert G\right\Vert _{\ast}\leq n^{3/2}/\sqrt{8}+n/2, \label{KMb}%
\end{equation}
with equality if and only if $G$ is the incidence graph of a particular type
of design.

Given the cases of equality in bounds (\ref{KM}) and (\ref{KMb}), arguably,
much of their thrill is in the fact that the bulk parameter \textquotedblleft
trace norm\textquotedblright\ is maximized on rare graphs of delicate structure.

To make the next step in this direction, recall that a graph is called
$r$\emph{-partite}\ if its vertices can be partitioned into $r$ edgeless sets.
We shall study the following natural problem arising in the vein of (\ref{KMb}):

\begin{problem}
\label{pr}If $r\geq3,$ what is the maximum trace norm of an $r$-partite graph
of order $n$?
\end{problem}

For complete $r$-partite graphs the question was answered in \cite{SGR15}, but
in general Problem \ref{pr} is much more difficult than the question for
bipartite graphs, for it has many variations, it requires novel constructions,
and most of it is beyond the reach of present methods.

First, we shall restate Problem \ref{pr} in analytic matrix form and shall
give some upper bounds. The matrix setup elucidates the main factors in the
graph problem. Further, using graph-theoretic proofs, we shall fine-tune these
upper bounds at the price of somewhat increased complexity.

We shall show that for infinitely many $r$ our bounds are exact or tight up to
low order terms. The intriguing point here is that the tightness of the bounds
is known only if $r$ is the order of a conference matrix, and since such
matrices do not exist for all $r,$ a lot of open problems arise.

\section{Upper bounds}

Given an $n\times n$ matrix $A=\left[  a_{i,j}\right]  $ and nonempty sets
$I\subset\left[  n\right]  $ and $J\subset\left[  n\right]  ,$ write $A\left[
I,J\right]  $ for the submatrix of all $a_{i,j}$ with $i\in I$ and $j\in J.$
An $n\times n$ matrix $A$ is called $k$\emph{-partite} if there is a partition
of its index set $\left[  n\right]  =N_{1}\cup\cdots\cup N_{k}$ such that
$A\left[  N_{i},N_{i}\right]  =0$ for any $i\in\left[  k\right]  $.

Further, write $A^{\ast}$ for the \emph{Hermitian transpose }of $A$, and let
$\left\Vert A\right\Vert _{\max}=\max_{i,j}\left\vert a_{i,j}\right\vert $. As
usual, $I_{n}$ and $J_{n}$ stand for the identity and the all-ones matrices of
order $n;$ we let $K_{n}=J_{n}-I_{n}$.

\begin{theorem}
\label{th1}Let $n\geq r\geq2,$ and let $A$ be an $n\times n$ complex matrix
with $\left\Vert A\right\Vert _{\max}\leq1.$ If $A$ is $r$-partite, then
\begin{equation}
\left\Vert A\right\Vert _{\ast}\leq n^{3/2}\sqrt{1-1/r}. \label{Mb}%
\end{equation}
Equality holds if and only if all singular values of $A$ are equal to
$\sqrt{\left(  1-1/r\right)  n}$.
\end{theorem}

\begin{proof}
Let $A=\left[  a_{i,j}\right]  $, and let $\sigma_{1},\ldots,\sigma_{n}$ be
the singular values of $A$. Clearly,
\begin{align*}
\left\Vert A\right\Vert _{\ast}^{2}  &  =\left(  \sigma_{1}+\cdots+\sigma
_{n}\right)  ^{2}\leq n\left(  \sigma_{1}^{2}+\cdots+\sigma_{n}^{2}\right)
=n(\mathrm{tr}(AA^{\ast}))\\
&  =\sum_{i,j\in\left[  n\right]  }\left\vert a_{i,j}\right\vert ^{2}\leq
n^{2}-\sum_{i\in\left[  r\right]  }\left\vert N_{i}\right\vert ^{2}\leq
n^{2}-\frac{1}{r}n^{2},
\end{align*}
completing the proof of (\ref{Mb}). If equality holds in (\ref{Mb}), then
\[
\left(  \sigma_{1}+\cdots+\sigma_{n}\right)  ^{2}=n\left(  \sigma_{1}%
^{2}+\cdots+\sigma_{n}^{2}\right)  =\left(  1-1/r\right)  n^{2},
\]
and so $\sigma_{1}=\cdots=\sigma_{n}=$ $\sqrt{\left(  1-1/r\right)  n}$,
completing the proof of Theorem \ref{th1}.
\end{proof}

\begin{remark}
A matrix $A=\left[  a_{i,j}\right]  $ that makes (\ref{Mb}) an equality has a
long list of further properties, e.g.: $r$ divides $n;$ the partition sets are
of size $n/r$; if an entry $a_{i,j}$ is not in a diagonal block, then
$\left\vert a_{i,j}\right\vert =1;$ and most importantly, $AA^{\ast}=\left(
1-1/r\right)  nI_{n}.$ Thus, the rows of $A$ are orthogonal, and so are its
columns. It seems hard to find for which $r$ and $n$ such matrices
exist.\medskip
\end{remark}

Next, from Theorem \ref{th1} we deduce a similar bound for nonnegative
matrices, in particular, for graphs.

\begin{theorem}
\label{th2}Let $n\geq r\geq2$, and let $A$ be an $n\times n$ nonnegative
matrix with $\left\Vert A\right\Vert _{\max}\leq1.$ If $A$ is $r$-partite,
then
\[
\left\Vert A\right\Vert _{\ast}\leq\frac{n^{3/2}}{2}\sqrt{1-1/r}+\left(
1-1/r\right)  n.
\]

\end{theorem}

\begin{proof}
For each $i\in\left[  r\right]  ,$ set $n_{i}=\left\vert N_{i}\right\vert ,$
and write $K$ for the matrix obtained from $J_{n}$ by zeroing $J\left[
N_{i},N_{i}\right]  $ for all $i\in\left[  r\right]  .$ Note that $K$ is the
adjacency matrix of the complete $r$-partite graph with vertex classes
$N_{1},\ldots,N_{r}.$ Since $K$ has no positive eigenvalue other than the
largest one $\lambda_{1}\left(  K\right)  $, we see that $\left\Vert
K\right\Vert _{\ast}=2\lambda_{1}\left(  K\right)  $. A result of
Cvetkovi\'{c} \cite{Cve72} implies that $\lambda_{1}\left(  K\right)
\leq\left(  1-1/r\right)  n,$ and so $\left\Vert K\right\Vert _{\ast}%
\leq2\left(  1-1/r\right)  n$.

Now, let $B:=2A-K,$ and note that the matrix $B$ and the sets $N_{1}%
,\ldots,N_{r}$ satisfy the premises of Theorem \ref{th1}; hence, using the
triangle inequality, we find that%
\[
n^{3/2}\sqrt{1-1/r}\geq\left\Vert B\right\Vert _{\ast}\geq\left\Vert
2A-K\right\Vert _{\ast}\geq2\left\Vert A\right\Vert _{\ast}-\left\Vert
K\right\Vert _{\ast}\geq2\left\Vert A\right\Vert _{\ast}-2\left(
1-1/r\right)  n,
\]
completing the proof of Theorem \ref{th2}.
\end{proof}

Note that the matrix $A$ in Theorems \ref{th1} and \ref{th2} needs not be
symmetric; nonetheless, the following immediate corollary gives precisely
Koolen and Moulton's bound (\ref{KMb}) if $r=2.$

\begin{corollary}
\label{cor1}Let $n\geq r\geq2.$ If $G\ $is an $r$-partite graph of order $n,$
then
\begin{equation}
\left\Vert G\right\Vert _{\ast}\leq\frac{n^{3/2}}{2}\sqrt{1-1/r}+\left(
1-1/r\right)  n. \label{eab}%
\end{equation}

\end{corollary}

\subsection{Upper bounds for graphs}

For $r\geq3$ bound (\ref{eab}) can be somewhat improved by more involved
methods. To this end, first we shall give an upper bound on the trace norm of
an $r$-partite graph with $n$ vertices and $m$ edges. Hereafter, $\lambda
_{i}\left(  G\right)  $ stands for the $i$'th largest eigenvalue of the
adjacency matrix of a graph $G.$

\begin{theorem}
\label{thum}Let $n>r>2$ and $2m\geq r^{2}n.$ If $G\ $is an $r$-partite graph
with $n$ vertices and $m$ edges, then
\begin{equation}
\left\Vert G\right\Vert _{\ast}\leq\frac{4m}{n}+\sqrt{\left(  n-r\right)
\left(  2m-\frac{r}{r-1}\left(  \frac{2m}{n}\right)  ^{2}\right)  }.
\label{bom}%
\end{equation}
Equality holds if and only if the following three conditions are met:

(i) $G$ is a regular graph;

(ii) the $r-1$ smallest eigenvalues of $G$ satisfy
\[
\lambda_{n}\left(  G\right)  =\cdots=\lambda_{n-r+2}\left(  G\right)
=-\frac{2m}{\left(  r-1\right)  n};
\]

(iii) the eigenvalues $\lambda_{2}\left(  G\right)  ,\ldots,\lambda
_{n-r+1}\left(  G\right)  $ satisfy
\[
\lambda_{2}^{2}\left(  G\right)  =\cdots=\lambda_{n-r+1}^{2}\left(  G\right)
=\frac{1}{n-r}\left(  2m-\frac{r}{r-1}\left(  \frac{2m}{n}\right)
^{2}\right)  .
\]

\end{theorem}

\begin{proof}
Let the graph $G$ satisfy the premises of the theorem, and for short, write
$\lambda_{i}$ for $\lambda_{i}\left(  G\right)  .$ Using the fact that
$\lambda_{1}^{2}+\cdots+\lambda_{n}^{2}=2m$ and the AM-QM inequality, we see
that
\begin{align*}
\left\Vert G\right\Vert _{\ast}  &  =\lambda_{1}+\sum_{i=2}^{n-r+1}\left\vert
\lambda_{i}\right\vert +\sum_{i=n-r+2}^{n}\left\vert \lambda_{i}\right\vert
\leq\lambda_{1}+\sum_{i=n-r+2}^{n}\left\vert \lambda_{i}\right\vert
+\sqrt{\left(  n-r\right)  \sum_{i=2}^{n-r+1}\lambda_{i}^{2}}\\
&  =\lambda_{1}+\sum_{i=n-r+2}^{n}\left\vert \lambda_{i}\right\vert
+\sqrt{\left(  n-r\right)  \left(  2m-\lambda_{1}^{2}-\sum_{i=n-r+2}%
^{n}\left\vert \lambda_{i}\right\vert ^{2}\right)  }\\
&  \leq\lambda_{1}+\sum_{i=n-r+2}^{n}\left\vert \lambda_{i}\right\vert
+\sqrt{\left(  n-r\right)  \left(  2m-\lambda_{1}^{2}-\frac{1}{r-1}\left(
\sum_{i=n-r+2}^{n}\left\vert \lambda_{i}\right\vert \right)  ^{2}\right)  }.
\end{align*}
Since $G$ is $r$-partite, Hoffman's bound \cite{Hof70} implies that
\[
\lambda_{1}\leq\left\vert \lambda_{n-r+2}\right\vert +\cdots+\left\vert
\lambda_{n}\right\vert .
\]
Now, letting $x=\lambda_{1},$ $y=\left\vert \lambda_{n-r+2}\right\vert
+\cdots+\left\vert \lambda_{n}\right\vert ,$ and $2m=A,$ we maximize the
function
\[
f\left(  x,y\right)  :=x+y+\sqrt{\left(  n-r\right)  \left(  A-x^{2}-\frac
{1}{r-1}y^{2}\right)  },
\]
subject to the constraints%
\begin{equation}
n\geq r\geq3,\text{ \ \ }A\geq r^{2}n,\text{ \ \ }y\geq x\geq A/n,\text{
\ \ }x^{2}+\frac{1}{r-1}y^{2}\leq A. \label{cons}%
\end{equation}
We shall show that $f\left(  x,y\right)  <f\left(  A/n,A/n\right)  $, unless
$y=x=A/n$. To this end, first we show that $f\left(  x,y\right)  $ is
decreasing in $y$ if $y>x.$ Assume the opposite, that is to say, there are $x$
and $y$, satisfying (\ref{cons}), with $x<y$ and .%
\[
\frac{\partial f\left(  x,y\right)  }{\partial y}=1-\frac{\left(  n-r\right)
y/\left(  r-1\right)  }{\sqrt{\left(  n-r\right)  \left(  A-x^{2}-\frac
{1}{r-1}y^{2}\right)  }}\geq0.
\]
After some algebra we obtain%
\begin{align*}
A  &  \geq\frac{1}{r-1}y^{2}+\left(  n-r\right)  \frac{y^{2}}{\left(
r-1\right)  ^{2}}+x^{2}>\left(  \frac{n-1}{\left(  r-1\right)  ^{2}}+1\right)
x^{2}>\left(  \frac{n-1}{\left(  r-1\right)  ^{2}}+1\right)  \frac{A^{2}%
}{n^{2}}\\
&  \geq\left(  \frac{n-1}{\left(  r-1\right)  ^{2}}+1\right)  r^{2}\frac{A}%
{n}>A,
\end{align*}
a contradiction, proving that $f\left(  x,y\right)  <$ $f\left(  x,x\right)
$, unless $y=x.$

Next, we maximize the function
\[
g\left(  x\right)  :=2x+\sqrt{\left(  n-r\right)  \left(  A-\frac{r}{r-1}%
x^{2}\right)  }\text{,}%
\]
subject to the constraints%
\begin{equation}
n>r\geq3,\text{ \ \ }A\geq r^{2}n,\text{ \ \ }x\geq A/n,\text{ \ \ }\frac
{r}{r-1}y^{2}\leq A. \label{cons1}%
\end{equation}
We shall show that $g\left(  x\right)  <g\left(  A/n\right)  $, unless
$x=A/n.$ To this end, we shall prove that $g\left(  x\right)  $ is decreasing
in $x,$ whenever $x>A/n.$ Assume the opposite, that is to say, there is an
$x,$ satisfying (\ref{cons1}), with $x<A/n$ and%
\[
\frac{dg\left(  x\right)  }{dx}=2-\frac{\left(  n-r\right)  rx/\left(
r-1\right)  }{\sqrt{\left(  n-r\right)  \left(  A-\frac{r}{r-1}x^{2}\right)
}}\geq0.
\]
After some algebra we get
\begin{align*}
4A  &  \geq\left(  \frac{\left(  n-r\right)  r^{2}}{\left(  r-1\right)  ^{2}%
}+\frac{r}{r-1}\right)  x^{2}>\left(  \frac{\left(  n-r\right)  r^{2}}{\left(
r-1\right)  ^{2}}+\frac{r}{r-1}\right)  \frac{A^{2}}{n^{2}}\geq\left(
\frac{\left(  n-r\right)  r^{2}}{\left(  r-1\right)  ^{2}}+\frac{r}%
{r-1}\right)  r^{2}\frac{A}{n}\\
&  =\left(  r-\frac{r^{2}-r+1}{n}\right)  \frac{r^{3}}{\left(  r-1\right)
^{2}}A>\left(  r-\frac{r^{2}-r+1}{r}\right)  \frac{r^{3}}{\left(  r-1\right)
^{2}}A=\frac{r^{2}}{r-1}A>4A.
\end{align*}
This contradiction implies that $g\left(  x\right)  <g\left(  A/n\right)  $,
unless $x=A/n.$ Therefore, $f\left(  x,y\right)  <f\left(  A/n,A/n\right)  $
unless $y=x=A/n$. This inequality implies (\ref{bom}). It also implies that if
equality holds in (\ref{bom}) then $\lambda_{1}=2m/n,$ and so clause
\emph{(i)} follows. Further, equality in (\ref{bom}) implies that%
\[
\lambda_{n-r+2}+\cdots+\lambda_{n}=-\frac{2m}{n}\text{ \ \ and \ \ }%
\sum_{i=n-r+2}^{n}\lambda_{i}^{2}=\frac{1}{r-1}\left(  \sum_{i=n-r+2}%
^{n}\left\vert \lambda_{i}\right\vert \right)  ^{2}%
\]
and so clause \emph{(ii) }follows as well. Finally, equality in (\ref{bom})
implies clause \emph{(iii)} in view of
\[
\sum_{i=2}^{n-r+1}\lambda_{i}^{2}=\frac{1}{n-r}\left(  \sum_{i=2}%
^{n-r+1}\left\vert \lambda_{i}\right\vert \right)  ^{2}\text{ \ \ and
\ \ }\sum_{i=2}^{n-r+1}\lambda_{i}^{2}=2m-\lambda_{1}^{2}-\sum_{i=n-r+2}%
^{n}\lambda_{i}^{2},
\]
completing the proof of Theorem \ref{thum}.
\end{proof}

Next, we maximize bound (\ref{bom}) over $m$ and get a bound that depends only
on $r$ and $n.$

\begin{theorem}
\label{thun}Let $r\geq2$ and $n\geq4\left(  r-1\right)  ^{2}.$ If $G\ $is an
$r$-partite graph of order $n,$ then%
\begin{equation}
\left\Vert G\right\Vert _{\ast}\leq\frac{n\left(  n-r\right)  }{2\sqrt{\left(
n-r\right)  \frac{r}{r-1}+4}}+\frac{\left(  r-1\right)  n}{r}+\frac{2\left(
r-1\right)  n}{r\sqrt{\left(  n-r\right)  \frac{r}{r-1}+4}}. \label{bon}%
\end{equation}
Equality holds if and only if the following three conditions are met:

(i) $G$ is a regular graph of degree
\[
\left(  1+\frac{2}{\sqrt{\left(  n-r\right)  \frac{r}{r-1}+4}}\right)
\frac{\left(  r-1\right)  n}{2r};
\]

(ii) the $r-1$ smallest eigenvalues of $G$ satisfy
\[
\lambda_{n}\left(  G\right)  =\cdots=\lambda_{n-r+2}\left(  G\right)
=-\left(  1+\frac{2}{\sqrt{\left(  n-r\right)  \frac{r}{r-1}+4}}\right)
\frac{n}{2r}.
\]

(iii) the eigenvalues $\lambda_{2}\left(  G\right)  ,\ldots,\lambda
_{n-r+1}\left(  G\right)  $ satisfy%
\[
\left\vert \lambda_{2}\left(  G\right)  \right\vert =\cdots=\left\vert
\lambda_{n-r+1}\left(  G\right)  \right\vert =\frac{n}{2\sqrt{\left(
n-r\right)  \frac{r}{r-1}+4}}.
\]

\end{theorem}

\begin{proof}
If $2m\geq r^{2}n,$ we maximize the function
\[
f\left(  x\right)  :=2x+\sqrt{\left(  n-r\right)  \left(  xn-\frac{r}%
{r-1}x^{2}\right)  },
\]
subject to
\[
\left(  r-1\right)  ^{2}/2\leq x\leq\left(  1-1/r\right)  n,
\]
and find that $f\left(  x\right)  $ attains a maximum for%
\[
x=\left(  1+\frac{2}{\sqrt{\left(  n-r\right)  \frac{r}{r-1}+4}}\right)
\frac{\left(  r-1\right)  n}{2r},
\]
which gives precisely (\ref{bon}).

If $2m<r^{2}n,$ we use the crude estimate by the AM-QM inequality,%
\[
\left\Vert G\right\Vert _{\ast}\leq\sum_{i=1}^{n}\left\vert \lambda
_{i}\right\vert <\sqrt{n\sum_{i=1}^{n}\lambda_{i}^{2}}=\sqrt{2mn}<rn.
\]
Now (\ref{bon}) follows by%
\[
r<\frac{\left(  n-r\right)  }{2\sqrt{\left(  n-r\right)  \frac{r}{r-1}+4}%
}+\frac{r-1}{r},
\]
which is equivalent to
\begin{equation}
\sqrt{\frac{r-1}{r-1/4}+\frac{\left(  r-1\right)  ^{2}}{r^{2}\left(
r-1/4\right)  ^{2}}}<\frac{r\left(  r-1\right)  }{r^{2}-r+1} \label{in1}%
\end{equation}
For the left side of (\ref{in1}) we get
\begin{align*}
\sqrt{\frac{r-1}{r-1/4}+\frac{\left(  r-1\right)  ^{2}}{r^{2}\left(
r-1/4\right)  ^{2}}}  &  <\sqrt{1-\frac{3}{4r-1}+\frac{1}{r^{2}}}<1-\frac
{3}{2\left(  4r-1\right)  }+\frac{1}{2r^{2}}\\
&  <1-\frac{3}{8r}+\frac{1}{2r^{2}}.
\end{align*}
For the right side of (\ref{in1}) we see that
\[
\frac{r\left(  r-1\right)  }{r^{2}-r+1}=1-\frac{1}{r^{2}-r+1}>1-\frac{1}%
{r^{2}-r}.
\]
Now, (\ref{in1}) follows from
\[
1-\frac{1}{r^{2}-r}>1-\frac{3}{8r}+\frac{1}{2r^{2}},
\]
which is true for $r\geq3.$

Clauses \emph{(i),(ii), }and \emph{(iii) }are\ just a restatement of last part
of Theorem \ref{thum}, so we omit them.
\end{proof}

\begin{remark}
\textbf{ }It is possible that bound (\ref{bon}) is exact for infinitely many
$r$ and $n.$ In general, it can be shown, that (\ref{bon}) is better than
(\ref{eab}) as long as $n>4\left(  r-1\right)  ^{2}$, but the difference
between their right sides never exceeds some constant that is independent of
$n.$ That is to say, (\ref{eab}) is exact within a linear term in $n.$
\end{remark}

\section{Constructions}

Recall that an \emph{Hadamard matrix} of order $n$ is an $n\times n$ matrix
$H$ with entries of modulus $1$ and such that $HH^{\ast}=nI_{n};$ hence, all
singular values of $H$ are equal to $\sqrt{n}$. Also, a \emph{conference
matrix} of order $n$ is an $n\times n$ matrix $C$ with zero diagonal, with
off-diagonal entries of modulus $1$, and such that $CC^{\ast}=\left(
n-1\right)  I_{n};$ hence all singular values of $C$ are equal to $\sqrt
{n-1}.$ For details on Hadamard and conference matrices the reader is referred
to \cite{CrKh07,IoKh07}. We shall write $\otimes$ for the Kronecker (tensor)
multiplication of matrices.\medskip\ 

First, we show that bound (\ref{Mb}) in Theorem \ref{th1} is best possible for
infinitely many $n,$ whenever $r$ is the order of a conference matrix.

\begin{theorem}
\label{th3}Let $r$ be the order of a conference matrix of order $r,$ and let
$k$ be the order of an Hadamard matrix. There exists an $r$-partite matrix $A$
of order $n=rk$ with $\left\Vert A\right\Vert _{\max}=1$ and
\[
\left\Vert A\right\Vert _{\ast}=n^{3/2}\sqrt{1-1/r}.
\]

\end{theorem}

\begin{proof}
Let $C\ $be a conference matrix of order $r$ and $H$ be an Hadamard matrix of
order $k.$ Let $A:=C\otimes H,$ and partition $\left[  rk\right]  $ into $r$
consecutive segments $N_{1},\ldots,N_{r}$ of length $k;$ we see that
$\left\Vert A\right\Vert _{\max}=1$ and that $A\left[  N_{i},N_{i}\right]  =0$
for any $i\in\left[  r\right]  .$ Finally, we find that
\[
\left\Vert A\right\Vert _{\ast}=\left\Vert C\otimes H\right\Vert _{\ast
}=\left\Vert C\right\Vert _{\ast}\left\Vert H\right\Vert _{\ast}=r\sqrt
{r-1}k^{3/2}=n^{3/2}\sqrt{1-1/r},
\]
completing the proof of Theorem \ref{th3}.
\end{proof}

Next, a modification of the above construction provides some matching lower
bounds for Theorems \ref{th2} and \ref{thun}, and Corollary \ref{cor1}.

\begin{theorem}
\label{th4}Let $r$ be the order of a real symmetric conference matrix. If $k$
is the order of a real symmetric Hadamard matrix, then there is an $r$-partite
graph $G$ of order $n=rk$ with
\[
\left\Vert G\right\Vert _{\ast}\geq\frac{n^{3/2}}{2}\sqrt{1-1/r}-\left(
1-1/r\right)  n.
\]

\end{theorem}

\begin{proof}
Let $C$ be a real symmetric conference matrix of order $r$, and let $H$ be a
real symmetric Hadamard matrix of order $k$. Let $B:=C\otimes H,$ and
partition $\left[  rk\right]  $ into $r$ consecutive segments $N_{1}%
,\ldots,N_{r}$ of length $k.$ We see that $B\left[  N_{i},N_{i}\right]  =0$
for any $i\in\left[  r\right]  ,$ and also $B\left[  N_{i},N_{j}\right]  $ is
a $\left(  -1,1\right)  $-matrix whenever $i,j\in\left[  r\right]  $ and
$i\neq j.$ Finally, let
\[
A:=\frac{1}{2}\left(  B+K_{r}\otimes J_{k}\right)  ,
\]
and note that $A$ is a symmetric $\left(  0,1\right)  $-matrix, and $A\left[
N_{i},N_{i}\right]  =0$ for any $i\in\left[  r\right]  $. Hence $A$ is the
adjacency matrix of an $r$-partite graph $G\ $of order $n.$ Note that the
singular values of $B$ are equal to $\sqrt{k\left(  r-1\right)  }%
=\sqrt{\left(  1-1/r\right)  n}.$ Thus, using the triangle inequality, we find
that
\[
\left\Vert \left(  B+K_{n}\otimes J_{k}\right)  \right\Vert _{\ast}%
\geq\left\Vert B\right\Vert _{\ast}-\left\Vert K_{r}\otimes J_{k}\right\Vert
_{\ast}\geq n^{3/2}\sqrt{1-1/r}-2\left(  r-1\right)  k,
\]
and so,
\[
\left\Vert G\right\Vert _{\ast}\geq\frac{n^{3/2}}{2}\sqrt{1-1/r}-\left(
1-1/r\right)  n,
\]
completing the proof of Theorem \ref{th4}.
\end{proof}

\begin{remark}
Complex Hadamard matrices of order $n$ exists for any $n.$ This is not true
for real Hadamard matrices, although there are various constructions of such
matrices, e.g., Paley's constructions:

If $q$ is an odd prime power, then there is a real conference matrix of order
$q+1,$ which is symmetric if $q=1$ $\left(  \operatorname{mod}4\right)  ;$
there is a real Hadamard matrix of order $q+1$ if $q=3$ $\left(
\operatorname{mod}4\right)  ;$ there is a real symmetric Hadamard matrix of
order $2\left(  q+1\right)  $ if $q=1$ $\left(  \operatorname{mod}4\right)  .$
\end{remark}

\section{Asymptotics}

Write $\mathcal{C}_{k}$ for the class of all complex $k$-partite matrices $A$
with $\left\Vert A\right\Vert _{\max}\leq1$, and let $\mathcal{H}_{k}\subset$
$\mathcal{C}_{k}$ be the subclass of the Hermitian matrices in $\mathcal{C}%
_{k}$. Likewise, write $\mathcal{R}_{k}$ for the class of the real $k$-partite
matrices $A$ with $\left\Vert A\right\Vert _{\max}\leq1$, and let
$\mathcal{S}_{k}\subset$ $\mathcal{R}_{k}$ be the subclass of the symmetric
elements of $\mathcal{R}_{k}$.

Further, write $n\left(  A\right)  $ for the order of a square matrix $A$, and
for any class of square matrices $\mathcal{X}$, let $\mathcal{X}\left(
n\right)  $ stand for the subclass of the elements of $\mathcal{X}$ with
$n\left(  A\right)  =n$.

With this notation let us define the functions $c_{k}\left(  n\right)  $,
$h_{k}\left(  n\right)  $, $r_{k}\left(  n\right)  $, and $s_{k}\left(
n\right)  $ as
\begin{align*}
c_{k}\left(  n\right)   &  :=\max\left\{  \left\Vert A\right\Vert _{\ast}%
:A\in\mathcal{C}_{k}\left(  n\right)  \right\}  \text{, \ \ \ \ }h_{k}\left(
n\right)  :=\max\left\{  \left\Vert A\right\Vert _{\ast}:A\in\mathcal{H}%
_{k}\left(  n\right)  \right\}  \text{,}\\
r_{k}\left(  n\right)   &  :=\max\left\{  \left\Vert A\right\Vert _{\ast}%
:A\in\mathcal{R}_{k}\left(  n\right)  \right\}  \text{, \ \ \ }s_{k}\left(
n\right)  :=\max\left\{  \left\Vert A\right\Vert _{\ast}:A\in\mathcal{S}%
_{k}\left(  n\right)  \right\}  \text{.}%
\end{align*}
Theorem \ref{th3} shows that if there is a complex conference matrix of order
$k$, then $c_{k}\left(  n\right)  =n^{3/2}\sqrt{1-1/r}$; similar statements
hold also for $h_{k}\left(  n\right)  $, $r_{k}\left(  n\right)  $, and
$s_{k}\left(  n\right)  $. However, conference matrices are rare and it is
difficult to determine $c_{k}\left(  n\right)  $, $h_{k}\left(  n\right)  $,
$r_{k}\left(  n\right)  $, and $s_{k}\left(  n\right)  $ for any $k$. Thus, in
what follows, we shall prove the possibility for certain asymptotics in $n$
for each of these functions.

For a start, Theorem \ref{th1} implies that if $A\in\mathcal{C}_{k}$, then
$\left\Vert A\right\Vert _{\ast}\leq\left(  n\left(  A\right)  \right)
^{3/2}$. Therefore, for any $k\geq2$, it is possible to define the constants
$c_{k}$, $h_{k}$, $r_{k}$, and $s_{k}$ as
\begin{align*}
c_{k}  &  :=\sup\left\{  \frac{\left\Vert A\right\Vert _{\ast}}{\left(
n\left(  A\right)  \right)  ^{3/2}}:A\in\mathcal{C}_{r}\right\}  \text{,
\ \ \ \ }h_{k}:=\sup\left\{  \frac{\left\Vert A\right\Vert _{\ast}}{\left(
n\left(  A\right)  \right)  ^{3/2}}:A\in\mathcal{H}_{r}\right\}  \text{,}\\
r_{k}  &  :=\sup\left\{  \frac{\left\Vert A\right\Vert _{\ast}}{\left(
n\left(  A\right)  \right)  ^{3/2}}:A\in\mathcal{R}_{r}\right\}  \text{,
\ \ \ }s_{k}:=\sup\left\{  \frac{\left\Vert A\right\Vert _{\ast}}{\left(
n\left(  A\right)  \right)  ^{3/2}}:A\in\mathcal{S}_{r}\right\}  \text{.}%
\end{align*}
Note again that Theorem \ref{th3} yields $c_{k}=\sqrt{1-1/r}$ if there is a
complex conference matrix of order $k$, and similar statements can be proved
also for $h_{k}$, $r_{k}$, and $s_{k}$. However, the main use of $c_{k}$,
$h_{k}$, $r_{k}$, and $s_{k}$ is to provide some asymptotics for $c_{k}\left(
n\right)  $, $h_{k}\left(  n\right)  $, $r_{k}\left(  n\right)  $, and
$s_{k}\left(  n\right)  $: indeed the above definitions imply that
\[
c_{k}\left(  n\right)  \leq c_{k}n^{3/2}\text{, \ \ }h_{k}\left(  n\right)
\leq h_{k}n^{3/2}\text{, \ \ }r_{k}\left(  n\right)  \leq r_{k}n^{3/2}\text{,
\ and\ \ }s_{k}\left(  n\right)  \leq s_{k}n^{3/2}\text{,}%
\]
and, as it turns out, these inequalities are tight.

\begin{theorem}
For any $k\geq2$, the functions $c_{k}\left(  n\right)  $, $h_{k}\left(
n\right)  $, $r_{k}\left(  n\right)  $, and $s_{k}\left(  n\right)  $ satisfy:%
\[
\lim_{n\rightarrow\infty}\frac{c_{k}\left(  n\right)  }{n^{3/2}}=c_{k}\text{,
\ \ }\lim_{n\rightarrow\infty}\frac{h_{k}\left(  n\right)  }{n^{3/2}}%
=h_{k}\text{, \ }\lim_{n\rightarrow\infty}\frac{r_{k}\left(  n\right)
}{n^{3/2}}=r_{k}\text{, \ and \ \ }\lim_{n\rightarrow\infty}\frac{s_{k}\left(
n\right)  }{n^{3/2}}=s_{k}\text{.}%
\]

\end{theorem}

\begin{proof}
We shall prove only the statement for $h_{k}\left(  n\right)  $, as the proofs
of the other cases are essentially the same. Let $\varepsilon>0$ and choose a
matrix $A\in\mathcal{H}_{k}$, say with $n\left(  A\right)  =n$, such that
\[
\frac{\left\Vert A\right\Vert _{\ast}}{n^{3/2}}>h_{k}-\frac{\varepsilon}%
{2}\text{.}%
\]
We shall show that for $m$ sufficiently large
\begin{equation}
\frac{h_{k}\left(  m\right)  }{m^{3/2}}>h_{k}-\varepsilon\text{.} \label{inb}%
\end{equation}
To this end, let $q$ be any prime with $q=1$ $\left(  \operatorname{mod}%
4\right)  $, and recall that Paley's construction yields a real symmetric
Hadamard matrix $H$ of order $2\left(  q+1\right)  $. Set $B:=A\otimes H$ and
note that $B$ is a Hermitian $k$-partite matrix with $\left\Vert B\right\Vert
_{\max}\leq1$. Hence $B\in\mathcal{H}_{k}\left(  2\left(  q+1\right)
n\right)  $. Thus, for any integer $m=2\left(  q+1\right)  n$, where $q$ is a
prime with $q=1$ $\left(  \operatorname{mod}4\right)  $, we see that
\begin{equation}
\frac{h_{k}\left(  m\right)  }{m^{3/2}}\geq\frac{\left\Vert B\right\Vert
_{\ast}}{\left(  n\left(  B\right)  \right)  ^{3/2}}=\frac{\left\Vert A\otimes
H\right\Vert _{\ast}}{n^{3/2}\left(  n\left(  H\right)  \right)  ^{3/2}}%
=\frac{\left\Vert A\right\Vert _{\ast}}{n^{3/2}}>h_{k}-\frac{\varepsilon}%
{2}\text{.} \label{inb1}%
\end{equation}

Now, let $m$ be any integer and let $q$ be the largest prime with $q=1$
$\left(  \operatorname{mod}4\right)  $ such that $2\left(  q+1\right)  n<m$.
Set $2\left(  q+1\right)  n:=t$ and let $C\in\mathcal{H}_{k}\left(  t\right)
$ be a matrix with $\left\Vert C\right\Vert _{\ast}=h_{k}\left(  t\right)  $.
Hence (\ref{inb1}) implies that
\[
\frac{\left\Vert C\right\Vert _{\ast}}{n^{3/2}\left(  2\left(  q+1\right)
\right)  ^{3/2}}>h_{k}-\frac{\varepsilon}{2}\text{.}%
\]
Let $\left[  t\right]  =N_{1}\cup\cdots\cup N_{k}$ be the partition of the
index set of $C$ such that $C\left[  N_{i},N_{i}\right]  =0$ for any
$i\in\left[  k\right]  $. Define an $m\times m$ matrix $B$, by extending $C$
with $m-t$ zero columns and rows and letting $N_{k}^{\prime}:=N_{k}\cup\left(
\left[  m\right]  \backslash\left[  t\right]  \right)  $. Thus, $B\left[
N_{k}^{\prime},N_{k}^{\prime}\right]  =0$, and so $B$ is $k$-partite. Clearly
$\left\Vert B\right\Vert _{\max}\leq1$, and therefore $B\in\mathcal{H}%
_{k}\left(  m\right)  $. Further, we find that
\[
\left\Vert B\right\Vert _{\ast}\geq\left\Vert C\right\Vert _{\ast}\geq\left(
h_{k}-\frac{\varepsilon}{2}\right)  n^{3/2}\left(  2\left(  q+1\right)
\right)  ^{3/2}\text{,}%
\]
and hence,
\[
\frac{\left\Vert B\right\Vert _{\ast}}{m^{3/2}}>\left(  h_{k}-\frac
{\varepsilon}{2}\right)  \frac{n^{3/2}\left(  2\left(  q+1\right)  \right)
^{3/2}}{m^{3/2}}\text{.}%
\]
A result about the distribution of primes \cite{BHP01} implies that if $q$ is
large enough, then there is a prime $p$ with $p=1$ $\left(  \operatorname{mod}%
4\right)  $ such that $q<p<q+q^{11/20}$. Hence
\[
m<2\left(  q+q^{11/20}+1\right)  n\text{,}%
\]
and so, using Bernoulli's inequality,%
\begin{align*}
\frac{\left\Vert B\right\Vert _{\ast}}{m^{3/2}}  &  >\left(  h_{k}%
-\frac{\varepsilon}{2}\right)  \frac{\left(  q+1\right)  ^{3/2}}{\left(
q+q^{11/20}+1\right)  ^{3/2}}>\left(  h_{k}-\frac{\varepsilon}{2}\right)
\left(  1-\frac{q^{11/20}}{q+1}\right) \\
&  >\left(  h_{k}-\frac{\varepsilon}{2}\right)  \left(  1-\frac{3}{2}%
q^{-9/20}\right)  \text{.}%
\end{align*}
Therefore, if $m$ is large enough, we see that (\ref{inb}) holds. Since
$h_{k}\left(  m\right)  \leq h_{k}m^{3/2}$, it follows that%
\[
\lim_{n\rightarrow\infty}\frac{h_{k}\left(  n\right)  }{n^{3/2}}=h_{k}\text{,}%
\]
completing the proof of the theorem.
\end{proof}

Particularly challenging is the following problem:

\begin{problem}
Find $c_{3}$.
\end{problem}

\subsection{Two general bounds on $\left\Vert G\right\Vert _{\ast}$}

Since the set of known conference and Hadamard matrices is quite sparse, the
following two explicit estimates may be useful.

\begin{proposition}
For any $r>2,$ there is an $n_{0}(r),$ such that if $n>n_{0}(r)$, then there
is an $r$-partite graph $G\ $of order $n$ with
\[
\left\Vert G\right\Vert _{\ast}>\frac{n^{3/2}}{2}(1-1/\sqrt{r})-n^{21/20}.
\]

\end{proposition}

For large $r$ this bound can be improved using results on prime distribution:

\begin{proposition}
For any sufficiently large $r,$ there is an $n_{0}(r),$ such that if
$n>n_{0}(r),$ then there is a graph $G$ of order $n,$ such that
\[
\left\Vert G\right\Vert _{\ast}>\frac{n^{3/2}}{2}\sqrt{1-1/(r-r^{11/20})}.
\]

\end{proposition}

\end{document}